\documentclass[11 pt]{amsart}
\usepackage{amsmath,amsthm,amsfonts,amssymb,amscd,amsrefs}
\usepackage[top=3cm, bottom=3cm,left=2cm, right=2cm]{geometry}
 \usepackage{enumitem}
\usepackage{centernot}
\usepackage{color,soul}
  \usepackage[linkcolor=red,citecolor=green,colorlinks=true]{hyperref}
\usepackage{mathrsfs}
\usepackage{mathtools}
\usepackage{etoolbox}
\usepackage{scalerel,stackengine}
\stackMath
\newcommand\widecheck[1]{%
\savestack{\tmpbox}{\stretchto{%
  \scaleto{%
    \scalerel*[\widthof{\ensuremath{#1}}]{\kern-.6pt\bigwedge\kern-.6pt}%
    {\rule[-\textheight/2]{1ex}{\textheight}}
  }{\textheight}%
}{0.5ex}}%
\stackon[1pt]{#1}{\scalebox{-1}{\tmpbox}}%
}
\theoremstyle{definition}
\newtheorem{defn}{Definition}[section]

\theoremstyle{remark}

\newtheorem{rmk}[defn]{Remark}
\theoremstyle{plain}
\newtheorem{thm}[defn]{Theorem}
\newtheorem{lem}[defn]{Lemma}
\newtheorem{cor}[defn]{Corollary}
\newtheorem{prop}[defn]{Proposition}

\DeclareMathOperator*{\spn}{span}

\DeclareMathOperator*{\conv}{conv}

\numberwithin{equation}{section}
\makeatletter 

\newcommand{\Rmnum}[1]{\expandafter\@slowromancap\romannumeral #1@}

\makeatother
\makeatletter
\def\imod#1{\allowbreak\mkern10mu({\operator@font mod}\,\,#1)}
\makeatother
\makeatletter 
\patchcmd{\@settitle}{\uppercasenonmath\@title}{}{}{}
\patchcmd{\@setauthors}{\MakeUppercase}{}{}{}
\patchcmd{\section}{\scshape}{}{}{}
\makeatother
\begin{document}
\selectfont
\title{Structure of sets of strong subdifferentiability in dual $L^1$-spaces}
\author[JAYANARAYANAN]{C. R. JAYANARAYANAN}
\address[C. R. JAYANARAYANAN]{Department of Mathematics, Indian Institute of Technology Palakkad, 678557, India\\
	\emph{E-mail~:} {\tt crjayan@iitpkd.ac.in}}
\author[RAO]{T. S. S. R. K. RAO}
\address[T. S. S. R. K. RAO]{Department of Mathematics, Ashoka University,
Sonipat - 131029, India\\	
        \emph{E-mail~:} {\tt srin@fulbrightmail.org}}

\begin{abstract}
	In this  article, we analyse the structure of finite dimensional subspaces of the set of points of strong  subdifferentiability in a dual space. In a dual $L_1(\mu)$ space, such a subspace is in the discrete part of the Yoshida-Hewitt type  decomposition. In this set up, any Banach space consisting of points of strong subdifferentiability  is necessarily finite dimensional. Our results also lead to streamlined and new proofs of results from  the study of strong proximinality for subspaces of finite co-dimension in a Banach space.
\end{abstract}
\maketitle
\section{Introduction}
Let $X$ be a real Banach space.  We recall from \cite{MR1216708} that a  vector $x \in X$  is said to be a point of strong subdifferentiability  (in short,  \textnormal{SSD}-point) if the one sided limit
 $\lim_{t \rightarrow 0^{+}} \frac{\|x+th\|-\|x\|}{t}$ exists uniformly for  $h\in B_X$, where $B_X$ denotes the closed unit ball of the Banach space $X$.
If $x^\ast \in X^\ast$ is a \textnormal{SSD} point, it was shown in \cite{MR1216708} that $x^\ast$ attains its norm.

To see the connection with proximinality, let  $Y \subset X$ be a closed subspace of $X$. For $x \in X$, let $P_Y(x) = \{y \in Y: d(x,Y) = \|x-y\|\}$.  An element of $P_Y(x)$ is called a best approximation to $x$ from $Y$. If $P_Y(x) \neq \emptyset$ for all $x \notin Y$, then $Y$ is said to be a proximinal subspace of $X$. It is well known that $x^\ast \in X^\ast$ attains its norm if and only if $Y = \ker(x^\ast)$ is a proximinal subspace of $X$.

In \cite{MR587395}, L. P. Vlasov introduced the notion of an H-set which is stronger than the notion of proximinality.
Later in \cite{MR1851725}, G. Godefroy and V. Indumathi introduced an equivalent notion of H-set, namely strong proximinality. We now give another equivalent definition of strongly proximinal subspace.
\begin{defn}
A proximinal subspace $Y$ of a Banach space $X$ is said to be strongly proximinal in $X$ if, for any $x \notin Y$ and for any sequence  $\{y_n\}_{n \geq 1}$ in $Y$ with $d(x,Y) = \lim_n\|x-y_n\|$ (such a sequence is called a minimizing sequence for $x$ in $Y$), there is a subsequence $\{y_{n_k}\}$ and a sequence $\{z_k\}_{k \geq 1}$ in $P_Y(x)$ such that $\|y_{n_k}-z_k\| \rightarrow 0$.
\end{defn}

 Our main tool in the investigation  of the structure of sets of \textnormal{SSD}-points is an astonishing geometric statement from \cite{MR1851725} that in a dual space $X^\ast$, a non-zero vector  $x^\ast$ in $X^*$
 is an \textnormal{SSD}-point of  $X^*$ if and only if $\ker(x^\ast)$ is a strongly proximinal subspace of $X$. Note that the definition of an \textnormal{SSD}-point makes no reference to the predual $X$. Thus a lot of information about \textnormal{SSD}-points in a dual space can be gleaned by choosing an appropriate predual of $X^\ast$.
 Using our analysis on \textnormal{SSD}-points, we give one of the most efficient proofs of the proximinality and strong proximinality of $M$-ideals in a Banach space without using the
 Bishop-Phelps theorem (as was done in \cite{MR2262920}) but only using simple geometry of \textnormal{SSD}-points in a dual space.

If $L_1(\mu)$ is a dual space and $F$  is a finite dimensional subspace of $L_1(\mu)$ consisting of \textnormal{SSD}-points of $L_1(\mu)$, then there is a positive integer $k$  such that $F \subset \ell_1(k)$ and $L_1(\mu) = \ell_1(k) \bigoplus_1 N$ for some closed subspace $N$, where $\ell_1(k)$ is the finite dimensional space $\mathbb{R}^k$ with $\ell_1$-norm. Moreover, we prove that any Banach space consisting of  \textnormal{SSD}-points of a dual $L_1(\mu)$ space  is finite dimensional.

 As an application  we show that if $X$ is a Banach space such that $X^\ast$ is isometric to $L_1(\mu)$ for some positive measure $\mu$
(such a Banach space $X$ is called an $L_1$-predual space), then in the class of subspaces of finite co-dimension, being strongly proximinal is transitive and is preserved under finite intersections.

We also show that the property of \textquoteleft finite intersection of strongly proximinal hyperplanes being strongly proximinal' is preserved under $c_0$-direct sums.

Our approach also leads to  several streamlined and some new proofs from  the study of strong proximinality for subspaces of finite co-dimension (see \cites{MR2262920,MR2308830, MR3314889, MR1851725}).

In this article, we consider only Banach spaces over the real field $\mathbb{R}$ and all subspaces we consider are
assumed to be closed. We denote the unit sphere of a Banach space $X$ by $S_X$. For an arbitrary collection $\{X_i: i \in I\}$  of Banach spaces, $\displaystyle{\sideset{}{_1}\bigoplus_{i\in I} X_i}$ denotes the $\ell_1$-direct sum equipped with the $\ell_1$-norm, and $\displaystyle{\sideset{}{_{c_0}}\bigoplus_{i\in I} X_i}$ and $\displaystyle{\sideset{}{_\infty}\bigoplus_{i\in I} X_i}$ denote the $c_0$-direct sum and $\ell_\infty$-direct sum  of $\{X_i: i \in I\}$ equipped with supremum norm  respectively. When there is no confusion, we omit writing the index set for the direct sums.

\section{Strong subdifferentiability and strong proximinality}
 Since finite dimensional subspaces of $X^*$ correspond precisely to finite co-dimensional subspaces of $X$, we first analyse the structure of finite co-dimensional strongly proximinal subspaces of $X$.   It is easy to see that any finite dimensional subspace  is a strongly proximinal subspace.  If $Y$ is a proximinal subspace of $X$ and $F$ is a finite dimensional subspace of $X$, then,  by \cite{MR874950}*{Corollary}, we have that $F+Y$ is a proximinal subspace of $X$ (see \cite{MR1000866} for more details). Now the following theorem from \cite{MR3904701} gives an analogous result for strong proximinality. As this is our principle tool, we reproduce the proof, for the sake of completeness.
\begin{thm}
	\label{sumsp}
Let $F$  be a finite dimensional subspace of a Banach space $X$ and $Y$ be a strongly proximinal subspace of $X$. Then $F+Y$ is a strongly proximinal subspace of $X$.
\end{thm}
\begin{proof}
We already have that $F+Y$ is a proximinal subspace of $X$. Now let $x \in X$. Let$\{f_n+y_n\}_{n \geq 1} \subset F+Y$ be a sequence such that $d(x,F+Y) = \lim_n \|x-f_n-y_n\|$, where
$\{f_n\}_{n \geq 1} \subset F$ and $\{y_n\}_{n \geq 1} \subset Y$. Since any minimizing sequence is bounded, and as $F+Y$ is closed, by \cite{MR1157815}*{Theorem 5.20},
we may assume without loss of generality that both $\{f_n\}_{n \geq 1}$ and $\{y_n\}_{n \geq 1}$ are bounded sequences. As $F$ is finite dimensional, by going through a subsequence, if necessary, we may and do assume that $f_n \rightarrow f_0$ for some $f_0 \in F$. We next claim that $d(x-f_0,Y) = d(x,F+Y) = \lim_n \|x-f_0-y_n\|$.

Clearly $|\|x-f_n-y_n\|-\|x-f_0-y_n\|| \leq \|f_n-f_0\|$ so that $\lim_n \|x-f_n-y_n\| = d(x,F+Y)= \lim_n\|x-f_0-y_n\| \geq d(x-f_0,Y)$. We always have $d(x,F+Y) \leq d(x-f_0,Y)$. Hence the claim. Since $Y$ is strongly proximinal in $X$,  we have a subsequence $\{y_{n_k}\}$ and a sequence $\{z_k\} \subset P_Y(x-f_0)$ such that $\|y_{n_k}-z_k\| \rightarrow 0$.

If $z \in P_Y(x-f_0)$, we have $\|x-f_0-z\|= d(x-f_0,Y) = d(x,F+Y)$, so that $f_0+z \in P_{F+Y}(x)$.

 Now $\{f_0+z_k\} \subset P_{F+Y}(x)$ and $\|f_{n_k}+y_{n_k} - (f_0+z_k)\| \rightarrow 0$. Thus $F+Y$ is a strongly proximinal subspace of $X$.
\end{proof}


\begin{cor} Let $f_1,...,f_n \in X^\ast$ and suppose $Y = \bigcap_{i=1}^n  ker(f_i)$ is a strongly proximinal subspace of $X$. Then, for any $g_1,...,g_k \in Y^\bot=\spn\{f_1,...,f_n\}$, $Z = \bigcap_{i=1}^k ker(g_i)$ is a strongly proximinal subspace of $X$.
\end{cor}
\begin{proof} Since $Y \subset Z \subset X$ and $Y$ is of finite co-dimension in $X$, we have that $Z = Y+F$ for some finite dimensional subspace $F$ of $X$. Thus
the conclusion follows from Theorem~\ref{sumsp}.
 \end{proof}
 Even though the  next result follows from \cite{MR1851725}*{Theorem~2.5}, we give an alternate proof of it using Theorem~\ref{sumsp}.
 \begin{cor}
 	\label{spinc}
 	Let $Z$ and $Y$ be subspaces of a Banach space $X$ such that $Z\subset Y\subset X$. If $Z$ is a finite co-dimensional strongly proximinal subspace of $X$, then $Y$ is  strongly proximinal in $X$.
 \end{cor}
 \begin{proof}
 	Since $Z$ is finite co-dimensional in $Y$, we have $Y=Z+ F$ for some finite dimensional subspace $F$ of $X$. As $Z$ is strongly proximinal in $X$ and  $F$ is finite dimensional, by Theorem~\ref{sumsp}, $Y$ is strongly proximinal in $X$.
 \end{proof}

%
%
 We now prove the following result from \cite{MR1851725} by using Theorem~\ref{sumsp}.
  \begin{cor}
 	\label{spssd}
 	Let $Y$ be a finite co-dimensional strongly proximinal subspace of a Banach space $X$. Then $Y^\perp\subset \{x^*\in X^*: x^*\mbox{ is an \textnormal{SSD}-point of }X^*\}$.
 \end{cor}
 \begin{proof}
 	Let $Y$ be a finite co-dimensional strongly proximinal subspace of $X$.
 	Let $f \in Y^\bot$. Then $Y \subset \ker(f) \subset X$. Since $Y$ is finite co-dimensional in $X$, $\ker(f) = Y + G$ for a finite dimensional subspace $G$ of $X$. Then, by Theorem~\ref{sumsp}, $\ker(f)$ is strongly proximinal in $X$ and hence, by \cite{MR1851725} , $f$ is an SSD-point of $X^*$. Thus $Y^\perp\subset \{x^*\in X^*: x^*\mbox{ is an \textnormal{SSD}-point of }X^*\}$.
 \end{proof}

 We now recall the notion of an $M$-ideal in a Banach space  which is known to be a  stronger notion than strong proximinality.
\begin{defn}[\cite{MR1238713}]
	Let $X$ be a Banach space. A linear projection P on $X$ is said to be an \emph{$L$-projection} if  $\|x\|=\|Px\|+\|x-Px\|$ for all $x \in X$. A subspace $Y$ of $X$ is said to be an $L$-summand in $X$ if it is the range of an $L$-projection. A subspace $Y$ of $X$ is said to be an \emph{$M$-ideal}  in $X$ if $Y^\bot$ is an $L$-summand in $X^*$.
\end{defn}

In what follows, we need the following result from  \cite{MR1238713}*{Chapter~\Rmnum{1}, Proposition~1.17}. We include a simple proof of this result for the sake of completeness.
\begin{prop}
	\label{midealinc}
	Let $Z$ and $Y$ be subspaces of a Banach space $X$ such that $Z\subset Y\subset X$. If $Z$ is an $M$-ideal in $X$, then $Z$ is an $M$-ideal in $Y$.
\end{prop}
\begin{proof}
	Let $P:X^* \to X^*$ be an $L$-projection such that $\ker(P)=Z^\perp$. It is easy to see that the projection $\widetilde{P}: Y^* \to Y^* $ defined by $\widetilde{P}(g)=P(f)|_Y$ is an $L$-projection with $\ker(P)=Z^\perp$, where $f$ is the Hahn-Banach extension of $g$ to $X$ and $Z^\perp$ is the annihilator of $Z$ in $Y^*$. Thus $Z$ is an $M$-ideal in $Y$.
\end{proof}
We next give a self contained proof of a classical result without invoking intersection properties of balls, as is usually done.
\begin{thm}
	\label{midealsp}
	Let $Y$ be an $M$-ideal in a Banach space $X$. Then $Y$ is strongly proximinal in $X$.
\end{thm}
\begin{proof}
	For all $x\in X$, by Proposition~\ref{midealinc}, $Y$ is an $M$-ideal in $\spn\{Y\cup \{x\}\}$. Since strong  proximinality of $Y$ in  $\spn\{Y\cup \{x\}\}$ for all $x\in X$ implies the strong proximinality of $Y$ in $X$, we assume without loss of generality that  $Y$ is an $M$-ideal of co-dimension one in $X$.
	
	Let $Y=\ker(f)$ for some $f\in S_{X^*}$. Then $\spn(f)=Y^\perp$ is an $L$-summand in $X^*$. Since $f$ is an SSD-point of $\spn(f)$ and $\spn(f)$ is an $L$-summand in $X^*$, by \cite{MR1216708}*{Proposition~2.1}, $f$ is an \textnormal{SSD}-point of $X^*$. Thus $Y=\ker(f)$ is strongly proximinal in $X$.
\end{proof}
The following result characterizes SSD-points of duals of $L_1$-predual spaces. For a compact Hausdorff space $K$, we denote by $C(K)$ the space of continuous functions on $K$ equipped with supremum norm
\begin{thm}
	\label{ssdl1}
	Let $X$ be an  $L_1$-predual space and $f\in S_{X^*}$.  Then $f$ is an SSD-point of $X^*$ if and only if $f=\sum_{i=1}^{n}\alpha_i f_i$ for some scalars $\alpha_i$ and for some distinct  $f_1, \ldots, f_n\in S_{X^*}$  such that $\spn(f_i)$ is an $L$-summand in $X^*$ for all $i$.
\end{thm}
\begin{proof}
	Suppose $f=\sum_{i=1}^{n}\alpha_i f_i$ for some distinct  $f_1, \ldots, f_n\in S_{X^*}$  such that $\spn(f_i)$ is an $L$-summand in $X^*$ for all $i$. Then $\ker(f_i)$ is an $M$-ideal in $X$ and hence, by \cite{MR1238713}*{Chapter~\Rmnum{1}, Proposition~1.11}, $\bigcap_{i=1}^n \ker(f_i)$ is an $M$-ideal in $X$. Since  $\bigcap_{i=1}^n \ker(f_i)\subset \ker(f)\subset X$, by  Theorem~\ref{midealsp} and Corollary~\ref{spinc} , $\ker(f)$ is strongly proximinal in $X$ and hence $f$ is an SSD-points of $X^*$.
	
	Now suppose $f$ is an SSD-point of $X^*$. Since $X$ is an $L_1$-predual space, by \cite{MR0493279}*{Section~21, Theorem~6}, $X^{**}$  is isometric to $C(K)$ for some compact Hausdorff space $K$ and by \cite{MR1238713}*{Chapter~\Rmnum{4}, Example~1.1}, $X^*$ is an $L$-summand in $X^{***}$. Thus $f$ is  an SSD-point of $X^{***}=C(K)^*$. Then, by \cite{MR2308830}*{Theorem~2.1}, $f$ is a finitely supported measure on $K$. So we may assume without loss of generality that $f=\sum_{i=1}^n \alpha_i \phi_i$ for $\alpha_i\in [0,1]$ with $\sum_{i=1}^n \alpha_i=1$ and for $\phi_i\in X^{***}$ with $\spn(\phi_i)$ is an $L$-summand in $X^{***}$ for all $i$.
	
	Now we note that for an $L$-summand $U$ in a Banach space $V$, if $u\in S_U$ is such that $u=\frac{v_1+v_2}{2}$ for unit vectors $v_1, v_2\in V$, then $v_1, v_2\in U$. For, if $P$ is an $L$-projection from $V$ onto $U$. Then   $u=P(u)=\frac{P(v_1)+P(v_2)}{2}$. Thus $\|P(v_1)\|=\|P(v_2)\|=\|v_1\|=\|v_2\|=1$. Since $P$ is an $L$-projection, we get $P(v_i)=v_i$ for all $i$ and hence $v_1, v_2\in U$.
	
	Since $X^*$ is an $L$-summand in $X^{***}$, by the above argument, we get $\phi_i\in X^*$ for all $i$. So if we let $f_i=\phi_i$ for all $i$, then $f_1, \ldots, f_n$ satisfy the required properties.
\end{proof}	
The next set of corollaries are easy to prove.
\begin{cor}
	\label{sumssd}
	Let  $X$ be an  $L_1$-predual space. Then the set of \textnormal{SSD}-points of $X^*$ is a subspace.
\end{cor}
\begin{cor}
Let $X$ be an $L_1$-predual space and $F\subset X^*$ be a finite dimensional subspace. Then $F\subset \{x^*\in X^*: x^*\mbox{ is an \textnormal{SSD}-point of }X^*\} $ if and only if $F\subset \spn\{x^*\in X^*: \spn(x^*) \mbox{ is an $L$-summand in } X^*\}$.
\end{cor}
The following proposition illustrates our remarks about using different preduals. For $C(K)$ spaces with $K$ being a dispersed compact Hausdorff space, we give a simple proof of  \cite{MR2308830}*{Theorem 2.1}.

\begin{prop} Let $K$ be a dispersed compact Hausdorff space and let  $\mu\in C(K)^\ast$ be a unit vector. Then $\mu$ is an SSD-point point of $C(K)^*$  if and only if support of $\mu$ is a finite set.
\end{prop}
\begin{proof} Suppose $\mu$ is an SSD-point of $C(K)^*$. Since $K$ is a dispersed compact Hausdorff space, we know that $C(K)^\ast$ is isometric to $\ell_1(\Gamma)$ for a discrete set $\Gamma$ (see  Sections 5 and 18 of \cite{MR0493279}).
	On the other hand if $c_0(\Gamma)$ denotes the space of functions on $\Gamma$ vanishing at $\infty$, then $c_0(\Gamma)^\ast = \ell_1(\Gamma)$. Since $ker(\mu)
	\subset c_0(\Gamma)$ is now a proximinal subspace, we have that $\mu$ attains its norm. It is now easy to see that $\mu$ has only finitely many non-zero coordinates. Thus $\mu$ is finitely supported as an element of $C(K)^\ast$.
	
	Conversely suppose support of $\mu$ is a finite set. Then $\mu=\sum_{i=1}^{n}\alpha_i \delta_{k_i}$ for scalars $\alpha_1, \ldots,\alpha_n$ and $k_1,\ldots,k_n\in K$. Now, by \cite{MR1238713}*{Chapter~I, Example~1.4(a)}, $\ker(\delta_{k_i})=\{f\in C(K): f(k_i)=0 \}$ is an $M$-ideal in $C(K)$. Then, by Theorem~\ref{midealsp}, $\ker(\delta_{k_i})$ is strongly proximinal in $C(K)$ and hence $\delta_{k_i}$ is an SSD-point of $C(K)^*$ for all $i$. Then, by Corollary~\ref{sumssd}, $\mu$ is an SSD-point of $C(K)^*$
\end{proof}

The next result gives the structure of finite dimensional subspaces of $L_1(\mu)$ when $L_1(\mu)$ is a dual space.
\begin{prop} Let $L_1(\mu)$ be a dual space and  $F$  be a finite dimensional subspace of $L_1(\mu)$. Then $F$ consists of \textnormal{SSD}-points of $L_1(\mu)$  if and only if there exists a  positive integer $k$ such that $F \subset \ell_1(k)$ and $L_1(\mu) = \ell_1(k) \bigoplus_1 N$ for a subspace $N$ of $L_1(\mu)$.
\end{prop}
\begin{proof}
	Suppose $F \subset \ell_1(k)$ and $L_1(\mu) = \ell_1(k)\bigoplus_1 N$. Since, for the canonical basis $\{e_1,\ldots,e_k\}$ of $\ell_1(k)$,
 $\spn(e_i)$  is an $L$-summand in $L_1(\mu)$, by Theorem~\ref{ssdl1}, we get that $F$ consists of \textnormal{SSD}-points of $L_1(\mu)$.

 Conversely suppose that $F$ consists of \textnormal{SSD}-points of $L_1(\mu)$. Then, again by Theorem~\ref{ssdl1}, there exists a positive integer $k$ such that $F \subset
 \spn \{f_1,...,f_k\}$, where $f_i$'s are linearly independent and $\spn(f_i)$ is an $L$-summand in $L_1(\mu)$. Clearly $\spn \{f_1,...,f_k\}$ is isometric to $\ell_1(k)$. Since $\spn \{f_1,...,f_k\}=\spn(f_1)+\ldots+\spn(f_k)$ and  sum of finitely many $L$-summands is again an $L$-summand, it follows that $\spn\{f_1,...,f_k\}$
  is an $L$-summand of $L_1(\mu)$.  Thus $L_1(\mu) = \ell_1(k) \bigoplus_1 N$ for a subspace $N$ of $L_1(\mu)$.
 \end{proof}
 \begin{cor} Let $L_1(\mu)$ be a dual space and let $B$ be a Banach space consisting of\textnormal{SSD}-points of $L_1(\mu)$. Then $B$ is a finite dimensional space.
 \end{cor}
 \begin{proof} By  Theorem~\ref{ssdl1}, $B \subset\spn\{f\in L_1(\mu): \spn(f) \mbox{ is an $L$-summand in } L_1(\mu)\}$.  Since  $B$ is a Banach space, by an application of the Baire category theorem, if $B$ is infinite dimensional, $B$ contains an infinite combination $\sum_1^{\infty}\alpha_i f_i$ where $\spn(f_i)$ is an $L$-summand for all $i$. But such a combination is not an \textnormal{SSD}-point. Thus $B$ is a finite dimensional space.
 \end{proof}

As another application of Theorem~\ref{ssdl1}, we prove the following  result  from \cite{MR3314889}*{Proposition~3.20}.
\begin{thm}
	\label{spssd1}
	Let $X$ be an $L_1$-predual space and $Y$ be a finite co-dimensional subspace of $X$. Then $Y$ is strongly proximinal in $X$ if and only if $Y^\perp\subset \{x^*\in X^*: x^*\mbox{ is an \textnormal{SSD}-point of }X^*\}$.
\end{thm}
\begin{proof}
	Suppose $Y^\perp\subset \{x^*\in X^*: x^*\mbox{ is an \textnormal{SSD}-point of }X^*\}$. Then, by Theorem~\ref{ssdl1}, $Y^\perp\subset \spn\{x^*\in X^*: \spn(x^*) \mbox{ is an $L$-summand in } X^*\}$. Since $Y^\perp$ is finite dimensional, there exists $g_1, \ldots, g_k\in X^*$ such that $\spn(g_i)$ is an $L$-summand in $X^*$ and $\bigcap_{i=1}^{k}\ker(g_i)\subset Y\subset X$. Since $\bigcap_{i=1}^{k}\ker(g_i)$ is an $M$-ideal  in $X$, the strong proximinality of $Y$ in $X$ follows from Theorem~\ref{midealsp} and Corollary~\ref{spinc}.
	
	The other implication follows from Theorem~\ref{spssd}.
\end{proof}
To indicate the technique of non-unique preduals, we briefly indicate an alternate proof of the above result in the case of separable Banach spaces.
\begin{proof}
	Let  $X$ be a separable Banach space and $X^\ast = L_1(\mu)$ for some positive measure $\mu$.
	
	If $X^\ast$ is also separable, then it is easy to see that the set of  extreme points of the unit ball of $X^*$ is a countable discrete set (for any two independent  extreme points $x^\ast_1$ and $x^\ast_2$, as $\spn\{x^\ast_1\}$ and $\spn\{x^\ast_2\}$ are $L$-summands with $\spn\{x^\ast_1\} \cap \spn\{x^\ast_2\}=\{0\}$, $\|x^\ast_1 \pm x^\ast_2\| = 2$). Thus, by \cite{MR0493279}*{Section 22, Page 226, Theorem 5}, $X^\ast = \ell_1$. If $c$ denotes the space of convergent sequences with supremum norm, then $c^\ast = \ell_1$.

If $X^\ast$ is not separable, then the set of extreme points of the  unit ball of $X^*$ is uncountable. Therefore, again  by \cite{MR0493279}*{Section 22, Page 226, Theorem 5}, $X^\ast = C([0,1])^\ast$.

 In either case $X^\ast$ is the dual of a $C(K)$ space for some compact Hausdorff space.
Therefore since \textnormal{SSD}-points are independent of the predual, the conclusion follows from \cite{MR2308830}*{Corollary 2.3}.
\end{proof}
 \begin{rmk} It may be recalled that a long standing open problem in the isometric theory is that given an $L_1$-predual space $X$, is there a $C(K)$ space
 such that $C(K)^\ast = X^\ast$.
 \end{rmk}

Now using these ideas related to \textnormal{SSD}-points, we give a simple proof of \cite{MR3314889}*{Theorem~3.10} in the case of $L_1$-predual spaces.

\begin{cor}
	Let $X$ be an $L_1$-predual space and $Y$ be a finite co-dimensional subspace of $X$. Then $Y$ is strongly proximinal in $X$ if and only if $Y^{\perp\perp}$ is strongly proximinal in $X^{**}$.
\end{cor}
\begin{proof}
	Suppose $Y$ is strongly proximinal in $X$. Let $Y=\bigcap_{i=1}^{n} \ker(f_i)$ for $f_1,\ldots, f_n\in X^*$. Then $Y^\perp=\spn\{f_i:i=1,\ldots,n\}$ is contained in set of all SSD-points of $X^*$. Since  $X^*$ is an $L$-summand in $X^{***}$, $F$ is contained in set all SSD-points of $X^{***}$. Since $Y^{\perp\perp}=\bigcap_{i=1}^{n} \ker(f_i)$ in $X^{**}=C(K)$ for some compact Hausdorff space $K$ and $Y^{\perp\perp\perp}=Y^{\perp}$ is contained in the set of all SSD-points of $X^{***}$, by Theorem~\ref{spssd1}, $Y^{\perp\perp}$ is strongly proximinal in $X^{**}$.
	
	Suppose $Y^{\perp\perp}$ is strongly proximinal in $X^{**}$. Again as above, let $Y=\bigcap_{i=1}^{n} \ker(f_i)$ for $f_1,\ldots, f_n\in X^*$. Then $Y^\perp=\spn\{f_i:i=1,\ldots,n\}$ in $X^*$.  Since $Y^{\perp\perp}=\bigcap_{i=1}^{n} \ker(f_i)$ is strongly proximinal in $X^{**}=C(K)$ for some compact Hausdorff space $K$, $Y^{\perp\perp\perp}$ is contained in set all of all SSD-points of $X^{***}$. Now as $Y^{\perp}$ is finite dimensional, $Y^\perp =Y^{\perp\perp\perp}$ and hence $Y^\perp$ is contained in set all SSD-points of $X^{***}$. Since $Y^\perp \subset X^*$, it follows that $Y^\perp$ is contained in set all SSD-points of $X^{*}$. Then, by Theorem~\ref{spssd1}, $Y$ is strongly proximinal in $X$.
\end{proof}
Our next lemma is motivated by the proof of \cite{MR2308830}*{Proposition~3.8}.
\begin{lem}
	\label{ssdext}
	Let $X$ be an $L_1$-predual space and $Y$ be a finite co-dimensional strongly proximinal subspace of $X$. If $\phi \in X^*$ is such that $\phi|_Y$ is a non-zero \textnormal{SSD}-point of $Y^*$, then $\phi$ is an \textnormal{SSD}-point of $X^*$.
\end{lem}
\begin{proof}
	Let $F=\phi|_Y$ and $F$ be an SSD-point of $Y^*$. Suppose $\{\varphi_1,\ldots,\varphi_n\}$ is a basis of $Y^\perp$. Then, by Theorem~\ref{ssdl1}, $\varphi_i$ is an SSD-point of $X^*$ for all $i$. Then for every $i$, there exist $f_{i1},\ldots,f_{in_i}\in X^*$ and scalars $\alpha_{i1},\ldots,\alpha_{i{n_i}}$ such that $\varphi_i=\sum _{j=1}^{n_i} \alpha_{ij}f_{ij}$ and $\spn\{f_{ij}\}$ is an $L$-summand in $X^*$ for all $j$. Thus $\ker(f_{ij})$ is an $M$-ideal in $X$ for all $i$ and $j$. Therefore $J=\bigcap_{i,j}\ker(f_{ij}) $ is an $M$-ideal in $X$.  Thus $X^*=J^\perp \oplus_1 J^{\#}$, where $J^{\#}=\{x^*\in X^*:\|x^*\|=\|x^*|_J\|\}$ and $J^{\#}$ is isometric to $J^*$ (see \cite{MR1238713}*{Chapter~I, Proof of Proposition ~1.12}). Since $J\subset Y\subset X$, by Proposition~\ref{midealinc}, $J$ is also an $M$-ideal in $Y$. Therefore $Y^*=J_1^\perp \oplus_1 J_1^{\#}$, where $J_1^\perp=\{y^*\in Y^*: y^*|_J=0\}$ and   $J_1^{\#}=\{y^*\in Y^*:\|y^*\|=\|y^*|_J\|\}$.  Observe that $J_1^{\#}$ is  isometric to $J^*$ as $J$ is an $M$-ideal in $Y$.
	
	Let $F=F_1+F_2$ for $F_1\in J_1^\perp$ and $F_2\in J_1^{\#}$. Then, by \cite{MR2308830}*{Lemma~3.7}, $F_1$ and $F_2$ are SSD-points of $J_1^\perp$ and $J_1^{\#}$ respectively. Since $F_2$ is an SSD-point of $J_1^{\#}$, $F_2|_J$ attains its norm on $J$.
	
	Let $\phi=\phi_1+\phi_2$ for $\phi_1\in J^\perp$ and $\phi_2\in J^{\#}$. Since $\phi_1\in \spn_{i,j}\{f_{ij}\}$, by Theorem~\ref{ssdl1}, $\phi_1$ is an SSD-point of $J^\perp$.
	It remains to show that $\phi_2$ is an SSD-point of $J_1^{\#}$.
	
	If $\phi_2$ is not an SSD-point of $J_1^{\#}$, then $\phi_2|_J$ is not an SSD-point of $J^*$ (see \cite{MR1238713}*{Chapter~I, Proof of Proposition ~1.12}). Then, by  \cite{MR1851725}*{Lemma~1.1} there exists an $\varepsilon>0$ and a sequence $(h_n)$ of unit vectors from $J$ such that $\phi_2(h_n)\to \|\phi_2\|$ and $d(h_n, \{x\in J: \|x\|=1 \mbox{ and } \phi_2(x)=\|\phi_2\|\})\ge \varepsilon$. Clearly $\phi_2(h)=\phi(h)=F_2(h)=F(h)$ for all $h\in J$. Hence $\lim_n F_2(h_n)=\lim_n \phi_2(h_n)=\|\phi_2\|=\|\phi_2|_J\|=\|F_2|_J\|=\|F_2\|$ and $d(h_n, \{x\in J: \|x\|=1 \mbox{ and } F_2(x)=\|F_2\|\})\ge \varepsilon$. Thus, by \cite{MR1851725}*{Lemma~1.1}, $F_2|_J$ is not an SSD-point of $J^*$. This contradicts the fact that $F_2$ is an SSD-point of $J_1^{\#}$. Thus $\phi_2$ is an SSD-point of $J^{\#}$ and hence, by \cite{MR2308830}*{Lemma~3.7}, $\phi$ is an SSD-point of $X^*$.
	\end{proof}

From Theorem~\ref{spssd1}, it follows that if $X$ is an $L_1$-predual space and $F\subset X^*$ is a finite dimensional subspace consisting  of SSD-points of $X^*$, then the pre-annihilator $Z = \bigcap_{f \in F}\ker(f)$ is a strongly proximinal subspace of $X$. We extend this result to any finite co-dimensional strongly proximinal subspaces of an $L_1$-predual space.

\begin{thm}
	\label{trans}
	Let $X$ be an $L_1$-predual space and $Y$ be a finite co-dimensional strongly proximinal subspace of $X$. Suppose $F \subset Y^*$ is a finite dimensional subspace consisting  of \textnormal{SSD}-points of $Y^*$. Then $Z = \bigcap_{f \in F} \ker(f)$ is a strongly proximinal subspace of $X$ and hence is a strongly proximinal subspace of $Y$.
\end{thm}
\begin{proof}
	Let $Z = \bigcap_{f \in F}\ker(f)$. Then $Z^\perp =F$.  Since $Y$ is of finite co-dimension in $X$,  without loss of generality we may assume that
	$Z = \bigcap_{1 \leq i \leq n}\ker f_i$ for some $f_i \in X^*$ and $f_i \notin Y^\perp$. Thus $f_i|_Y \in Z^\perp=F$ and hence $f_i|_Y$ is an \textnormal{SSD}-point of $Y^*$.
	
	Now by Lemma~\ref{ssdext}, we get that $f_i$ is an \textnormal{SSD}-point of $X^\ast$ as well. Thus, by Corollary~\ref{sumssd}, annihilator $Z^\perp$ of $Z$ in $X^*$ consists of SSD-points of $X^*$  and hence, by Theorem~\ref{spssd1}, we get that Z is strongly proximinal in X and hence a strongly proximinal subspace of $Y$.
\end{proof}

The following result is an immediate consequence of Theorem~\ref{trans}.
\begin{cor} Let $X$ be an $L_1$-predual space. Then being strongly proximinal is a transitive property for finite co-dimensional closed subspaces  of $X$. Moreover, intersection of finitely many strongly proximinal subspaces of finite co-dimension is again a strongly proximinal subspace.
\end{cor}
An interesting question that arises from Theorem~\ref{trans} is that in an $L_1$-predual space $X$,  if one can characterize strongly proximinal subspaces
of finite co-dimension that are again $L_1$-preduals. Our next theorem answers this question. Since being strongly proximinal is a relative property of the subspace $Y$ in $X$ and being an $L_1$-predual is a geometric property of the individual space $Y$, our surprising answer is based on the choice of a specific type of basis for $Y^\bot$.

\begin{thm}
Let $X$ be an $L_1$-predual space and  $Y$ be a  strongly proximinal subspace of finite co-dimension $n$ in $X$. Let $\{\Lambda_1,\ldots,\Lambda_n\}\subset S_{X^*}$
be a basis of  $Y^\bot \subset \ell^1(k) $ for some $k>0$ and $X^\ast = \ell^1(k)\bigoplus_1 N$ for some closed subspace $N$ of $X^*$. 
Suppose $\Lambda_i \in \conv\{\pm e_j:  \mbox{both } e_j \mbox{ and } -e_j ~\mbox{ do not occur simultaneously} \\\mbox{ in the convex combination},~1 \leq j\leq k\}$ for $1 \leq i \leq n$,  where $\conv$ denotes the convex hull. Assume that in the convex combination corresponding to each $\Lambda_i$, there is an $e_j$ or $-e_j$, which is different for each $i$ and  the coefficients of such $e_j$ or $-e_j$ in the convex combination for each $\Lambda_i$ are also distinct and are in $(\frac{1}{2},1]$. Then $Y$ is an $L_1$-predual space.

Conversely if $Y$ is a strongly proximinal subspace of finite codimension in $X$ and is also an  $L_1$-predual space, then  $Y^\bot$ admits a basis of the above form.
\end{thm}
\begin{proof}
To fix ideas, we will first work with $n = 1$ and then indicate how the general result follows.

Let $\Lambda = \sum_1^k \lambda_i e_i  \in \ell^1(k)$, for some $\lambda_i \in [0,1]$ with $\sum_1^k \lambda_i = 1 = \|\Lambda\|$.
Assume that  for some $i_0 >0$, $\lambda_{i_0} > \frac{1}{2}$.

We recall that a Banach space is an $L_1$-predual space if and only if it's bidual is isometric to a space of continuous functions on a compact set.
We will first show that   $ker(\Lambda)$ is the range of a projection of norm one in $X^{\ast\ast} = \ell_{\infty}(k) \bigoplus_{\infty} N^\ast$. This is clearly equivalent to showing $ker (\lambda_1,..., \lambda_k) \subset \ell_{\infty}(k)$ is the range of a projection of norm one in $\ell_{\infty}(k)$.
Since $\lambda_{i_0} > \frac {1}{2}$,  it follows from the Theorem of   \cite{MR960402} that $ker(\lambda_1,...,\lambda_k)$ is the range of a projection of norm one in $\ell_{\infty}(k)$.  Since $X$ is an $L_1$-predual space and $ker(\Lambda)\subset X^{**}$  is the range of a  projection of norm one in $X^{**}$, it follows from \cite{MR0493279}*{Theorem~6 of Section~21 and Theorem~2 of Section~11} that $ker(\Lambda) \subset  X$  is an $L^1$-predual space.

The converse part of the theorem is also proved in a similar way, by first noting that if $Y$ is an $L^1$-predual space, then $Y^{\bot\bot}$, being a dual space and a $C(K)$-space for some compact Hausdorff space $K$, is the range of a projection of norm one in $X^{**}= \ell^\infty (k) \bigoplus_{\infty} N^*$ (see \cite{MR0493279}*{Section~11 and Section~21}). By the structure of $Y^{\bot\bot}$ , this again implies that  $Y^{\bot\bot} \cap \ell^\infty (k)$ is the range of a projection of norm one in $\ell^\infty (k)$. Thus the choice of a basis for $Y^{\bot\bot}$ follows from the Theorem of   \cite{MR960402}.

Now if $Y$ is of codimension $n$, from our earlier analysis we have reduced the problem to conditions under which finite dimensional subspaces of $\ell_{\infty}(k)$ are ranges of projections of norm one. By again appealing to the Theorem of   \cite{MR960402}, we see that the conditions in the statement of the theorem are equivalent to the requirement for ranges of projections of norm one in $\ell_\infty(k)$.
\end{proof}
The next corollary follows from the arguments given above, in conjugation with our earlier results.

\begin{cor}
Let $X$ be an $L_1$-predual space. Then intersection of finitely many subspaces of finite codimension which are both $L_1$-preduals and strongly proximinal subspaces of $X$, is again an $L_1$-predual and a strongly proximinal subspace of $X$.
\end{cor}

In the next theorem, we show that the property of \textquoteleft finite intersection of strongly proximinal hyperplanes being strongly proximinal' is preserved under $c_0$-direct sums. We note that this property is equivalent to saying that  \textquoteleft intersection of kernels of  finitely many SSD-points is a strongly proximinal subspace'.

\begin{thm} Let $\{X_i\}_{i \in I}$ be a family of Banach spaces such that for every $X_i$ and for any collection of finitely many \textnormal{SSD}-points of $X_i^*$, the intersection of their kernels is a strongly proximinal subspace of $X_i$. Then the same conclusion holds for $X = \bigoplus_{c_0} X_i $.
%
%
\end{thm}

\begin{proof}
	We first prove the case when $I=\{1,2\}$. Let $F_1, F_2\in X^*$ be \textnormal{SSD}-points of $X^*= X_1^*\bigoplus_1 X_2^*$. Then $F_1=(F_1(1), F_1(2))$ and $F_2=(F_2(1), F_2(2))$, where $F_1(i)$ and $F_2(i)$ are either \textnormal{SSD}-points of $X_i^*$ or $0$  for $i=1,2$.
	
	Now consider the case where all the functionals are non-zero. We can  easily see that $$
		 \big(\ker(F_1(1))\cap \ker(F_2(1))\big)\sideset{}{_\infty}\bigoplus \big( \ker(F_1(2))\cap \ker(F_2(2))\big)\subset \ker(F_1)\cap \ker(F_2)\subset X.$$
		Since  $\ell_\infty$-sum of finitely many strongly proximinal subspaces is strongly proximinal, it is easy to see using the hypothesis that $\big(\ker(F_1(1))\cap \ker(F_2(1))\big)\sideset{}{_\infty}\bigoplus \big( \ker(F_1(2))\cap \ker(F_2(2))\big)$ is a finite co-dimensional strongly proximinal subspace of $X$ and hence, by Corollary~\ref{spinc}, $\ker(F_1)\cap \ker(F_2)$ is strongly proximinal in $X$.
		
		Similar arguments can be used to prove other cases and also to prove the case of finitely many SSD-points of $X^*$.
		
		Now let $F_1,\ldots, F_k\in X^*=\displaystyle{\sideset{}{_1}\bigoplus X_i^*} $ be SSD-points of $X^*$. Then, by \cite{MR3196182}*{Theorem~3.10}, there exists a finite subset $A$ of $I$ such that $F_i(j)=0$ for all $i=1,\ldots,k$ and $j \centernot \in A$. Then we may assume that $F_1,\ldots,F_k \in \displaystyle{\sideset{}{_1}\bigoplus_{i \in A}X_i^*}$.
		
		Then by arguments  similar to the ones given during the proof of the first case, we get a finite co-dimensional strongly proximinal subspace $Z$ of $\displaystyle{\sideset{}{_\infty}\bigoplus_{i \in A}X_i}$ such that 		
			\[
		 \left	 (Z\sideset{}{_\infty}\bigoplus \left(\sideset{}{_{c_0}}\bigoplus_{i \centernot\in A}X_i \right)\right) \subset \bigcap_{i=1}^k \ker(F_i)\subset X.
			\]	
	Hence, by Corollary~\ref{spinc}, $\bigcap_{i=1}^k \ker(F_i)$ is strongly proximinal in $X$.	
	\end{proof}

\subsection*{Acknowledgements}
The authors would like to thank  Journal of Convex Analysis for efficient handling of our article during the pandemic.

 

\begin{bibdiv}
 	\begin{biblist}		
 		\bib{MR960402}{article}{
 			author={Baronti, Marco},
 			title={{Norm-one projections onto subspaces of {$l^\infty$}}},
 			date={1988},
 			ISSN={0003-889X},
 			journal={Arch. Math. (Basel)},
 			volume={51},
 			number={3},
 			pages={242\ndash 246},
 		}
 		\bib{MR2308830}{article}{
 			author={Dutta, S.},
 			author={Narayana, Darapaneni},
 			title={{Strongly proximinal subspaces of finite codimension in {$C(K)$}}},
 			date={2007},
 			ISSN={0010-1354},
 			journal={Colloq. Math.},
 			volume={109},
 			number={1},
 			pages={119\ndash 128},
 			url={http://dx.doi.org/10.4064/cm109-1-10},
  		}
 		
 		\bib{MR874950}{article}{
 			author={Feder, Moshe},
 			title={{On the sum of proximinal subspaces}},
 			date={1987},
 			ISSN={0021-9045},
 			journal={J. Approx. Theory},
 			volume={49},
 			number={2},
 			pages={144\ndash 148},
 			url={https://doi.org/10.1016/0021-9045(87)90084-0},
 		}
 		
 		\bib{MR1216708}{article}{
 			author={Franchetti, Carlo},
 			author={Pay{\'a}, Rafael},
 			title={{Banach spaces with strongly subdifferentiable norm}},
 			date={1993},
 			journal={Boll. Un. Mat. Ital. B (7)},
 			volume={7},
 			number={1},
 			pages={45\ndash 70},
 		}
 		
 		\bib{MR1851725}{article}{
 			author={Godefroy, G.},
 			author={Indumathi, V.},
 			title={{Strong proximinality and polyhedral spaces}},
 			date={2001},
 			ISSN={1139-1138},
 			journal={Rev. Mat. Complut.},
 			volume={14},
 			number={1},
 			pages={105\ndash 125},
 		}
 		
 		\bib{MR1238713}{book}{
 			author={Harmand, P.},
 			author={Werner, D.},
 			author={Werner, W.},
 			title={{{$M$}-ideals in {B}anach spaces and {B}anach algebras}},
 			series={{Lecture Notes in Mathematics}},
 			publisher={Springer-Verlag},
 			address={Berlin},
 			date={1993},
 			volume={1547},
 			ISBN={3-540-56814-X},
 		}
 		
 		\bib{MR2262920}{article}{
 			author={Indumathi, V.},
 			author={Lalithambigai, S.},
 			title={{A new proof of proximinality for {$M$}-ideals}},
 			date={2007},
 			ISSN={0002-9939},
 			journal={Proc. Amer. Math. Soc.},
 			volume={135},
 			number={4},
 			pages={1159\ndash 1162},
 			url={https://doi.org/10.1090/S0002-9939-06-08701-6},
 		}
 		
 		\bib{MR3196182}{article}{
 			author={Jayanarayanan, C.~R.},
 			title={{Proximinality properties in {$L_p(\mu,X)$} and polyhedral direct
 					sums of {B}anach spaces}},
 			date={2014},
 			ISSN={0163-0563},
 			journal={Numer. Funct. Anal. Optim.},
 			volume={35},
 			number={6},
 			pages={708\ndash 723},
 			url={http://dx.doi.org/10.1080/01630563.2013.841195},
 		}
 		
 		\bib{MR3314889}{article}{
 			author={Jayanarayanan, C.~R.},
 			author={Paul, Tanmoy},
 			title={{Strong proximinality and intersection properties of balls in
 					{B}anach spaces}},
 			date={2015},
 			ISSN={0022-247X},
 			journal={J. Math. Anal. Appl.},
 			volume={426},
 			number={2},
 			pages={1217\ndash 1231},
 			url={http://dx.doi.org/10.1016/j.jmaa.2015.01.013},
 		}
 		
 		\bib{MR0493279}{book}{
 			author={Lacey, H.~Elton},
 			title={{The isometric theory of classical {B}anach spaces}},
 			publisher={Springer-Verlag},
 			address={New York},
 			date={1974},
 			note={Die Grundlehren der mathematischen Wissenschaften, Band 208},
 		}
 		
 		\bib{MR1000866}{article}{
 			author={Lin, Pei-Kee},
 			title={{A remark on the sum of proximinal subspaces}},
 			date={1989},
 			ISSN={0021-9045},
 			journal={J. Approx. Theory},
 			volume={58},
 			number={1},
 			pages={55\ndash 57},
 			url={https://doi.org/10.1016/0021-9045(89)90007-5},

}
 		
 		\bib{MR3904701}{article}{
 			author={Rao, T. S. S. R.~K.},
 			title={{Points of strong subdifferentiability in dual spaces}},
 			date={2018},
 			ISSN={0362-1588},
 			journal={Houston J. Math.},
 			volume={44},
 			number={4},
 			pages={1221\ndash 1226},
 		}
 		
 		\bib{MR1157815}{book}{
 			author={Rudin, Walter},
 			title={{Functional analysis}},
 			edition={Second},
 			series={{International Series in Pure and Applied Mathematics}},
 			publisher={McGraw-Hill, Inc., New York},
 			date={1991},
 			ISBN={0-07-054236-8},
 		}
 	
 		
 			\bib{MR587395}{article}{
 				author={Vlasov, L.~P.},
 				title={{Approximate properties of subspaces of finite codimension in
 						{$C(Q)$}}},
 				date={1980},
 				ISSN={0025-567X},
 				journal={Mat. Zametki},
 				volume={28},
 				number={2},
 				pages={205\ndash 222, 318},
 			}
 		
 		
 	\end{biblist}
 \end{bibdiv}
\end{document}